\documentclass[11pt,twoside]{amsart}

\usepackage[latin1]  {inputenc}%
\usepackage[T1]      {fontenc }%
\usepackage          {amsmath }%
\usepackage          {amsfonts}%
\usepackage          {amssymb }%
\usepackage          {amsthm  }%
\usepackage          {a4wide  }%
\usepackage          {url     }%
\usepackage          {tikz    }%
\usepackage[bookmarks=false,pdfborder={0 0 0.05}]{hyperref}
\usepackage[all]{xy}

\usepackage{enumerate, amsmath, amsfonts, amssymb, amsthm,  wasysym, graphics, graphicx, xcolor, frcursive,comment,bbm}

\newcommand{\R}{\mathbb{R}}
\newcommand{\Z}{\mathbb{Z}}

\DeclareMathOperator{\spanz}{span_{\mathbb{Z}}}
\DeclareMathOperator{\spanr}{span_{\mathbb{R}}}

\DeclareMathOperator{\idop}{id}

\DeclareMathOperator{\GL}{GL}

\DeclareMathOperator{\ZZ}{\mathbb{Z}}
\DeclareMathOperator{\RR}{\mathbb{R}}
\DeclareMathOperator{\NN}{\mathbb{N}}

\newtheorem{theorem}{Theorem}[section]
\newtheorem{corollary}[theorem]{Corollary}

\newtheorem{Lemma}[theorem]{Lemma}

\theoremstyle{definition}

\newtheorem{remark}[theorem]{Remark}
\newtheorem{example}[theorem]{Example}

\author[B.~Baumeister]{Barbara Baumeister}
\address{Barbara Baumeister, Universit\"at Bielefeld, Germany}
\email{b.baumeister@math.uni-bielefeld.de}
\thanks{}

\title{A note on Weyl groups and  root lattices}
\author[P.~Wegener]{Patrick Wegener}
\address{Patrick Wegener, Technische Universit\"at Kaiserslautern, Germany}
\email{wegener@mathematik.uni-kl.de}
\thanks{}

\begin{document}

\begin{abstract}
We follow the dual approach to Coxeter systems and show for Weyl groups that a  set of reflections generates  the group  if and only if the related sets of roots and coroots generate the root and
the coroot lattices, respectively.
Previously, we have proven if $(W,S)$ is a Coxeter system of finite rank $n$ with set of reflections $T$ and if $t_1, \ldots t_n \in T$ are reflections in $W$ that  generate $W$ then $P:=  \langle t_1, \ldots t_{n-1}\rangle$ is a parabolic subgroup of $(W,S)$ of rank $n-1$ \cite[Theorem~1.5]{BGRW16}. Here we show if $(W,S)$ is crystallographic  as well then all the reflections $t \in T$  such that $\langle P, t\rangle = W$ form a single orbit under  conjugation  by $P$.
\end{abstract}
\maketitle

\tableofcontents

\section{Introduction}

If a group $W$, $|W| >2$, contains a set of involutions $S$ such that $(W,S)$ is a Coxeter system then
the simple system $S$ is in general not uniquely determined by $W$. The problem to determine
all the Coxeter systems for which the possible simple systems $S$ are conjugate in $W$ is intensely studied under the name isomorphism problem for Coxeter groups and  is still open, see  \cite{Mue05} and \cite{MN16}.
Let $T:= \{ wsw^{-1} \mid w\in W, s \in S \}$ be the set of reflections of $(W,S)$. A special case of the isomorphism problem is the question which groups are strongly reflection rigid, that is for which Coxeter groups all
the simple systems that are contained in $T$ are conjugate in $W$. Recently there has been made
substantial progress (e.g. see \cite{CP10} or \cite{HP17}).

In this note we study finite Weyl groups, i.e. spherical crystallographic Coxeter systems $(W,S)$, and the related dual Coxeter systems $(W,T)$.
The dual approach, started independently by Brady and Watt~\cite{BW02} and by Bessis \cite{Be03}, is the study of Coxeter systems  via their set of reflections $T$. Clearly, it is also of interest
not only to study simple systems in $T$ but more generally minimal generating subsets
$X \subseteq T$ of $W$.  There is a $2$--$1$ map between the set of  roots $\Phi$  related to  $(W,S)$ and the set $T$ sending a root $\alpha$ to a reflection $s_\alpha$
(for the notation see the next section).   Let $X$ be a subset of $T$ and   $R = R(X)$ the related set of roots.

We  first  address the question whether it is possible to impose a natural necessary and sufficient condition on $R$ that characterises  the subsets $X$ of $T$ that
generate $W$. There is a positive answer. Let $L(R)$ be the $\Z$-span of the set of vectors $R$,  $R^{\vee}$ the set of coroots of $R$ and $W_R$ be the subgroup of $W$
 that is generated by the set of reflections $\{s_\alpha~|~\alpha \in R\}$.
We show the following.

\begin{theorem} \label{thm:GenGroupLattice}
Let $\Phi$ be a crystallographic root system of rank $n$, $R \subseteq \Phi$ non-empty and $\Phi'$ a root subsystem of $\Phi$ containing $R$. Then the following statements are equivalent:
\begin{enumerate}
\item[(a)] $\Phi' = W_R(R)$
\item[(b)] $W_{\Phi'}= W_R$
\item[(c)] $L(\Phi')= L(R)$ and  $L((\Phi')^{\vee})= L(R^\vee)$.
\end{enumerate}
\end{theorem}

\begin{corollary}\label{Corollary}
Let $(W,S)$ be a spherical crystallographic Coxeter system with root system $\Phi$ and set of reflections $T$. Then the following holds.
\begin{itemize}
\item[(a)] A subset $X$ of $T$ generates $W$ if and only if the related set of roots $R= R(X)$ generates  the root lattice $L(\Phi)$ and
$R^\vee$  generates the coroot lattice $L(\Phi^\vee)$.
\item[(b)]
A generating subset $X$ of $T$ is a minimal generating subset if and only if
$R$ and $R^\vee$ are basis of $L(\Phi)$ and $L(\Phi^\vee)$, respectively.
\end{itemize}
\end{corollary}

\begin{remark}
If $\Phi$ is  a non-crystallographic root system, then the $\Z$-span of $\Phi$ is either not a lattice  in the $\R$-span of $\Phi$ or it is a lattice, but $\Phi$  does not contain a basis of it.  Thus the equivalences in Theorem~\ref{thm:GenGroupLattice} do not hold for non-crystallographic root systems.
\end{remark}

Theorem~\ref{thm:GenGroupLattice} is a generalisation of \cite[Proposition 5.10 and Lemma 5.12]{BGRW16}.  There the simply laced case has been studied only.
In \cite{BGRW16}  also the concept of a quasi-Coxeter element has been introduced. An element $w \in W$ is a {\em quasi-Coxeter element}, if it has a reduced factorization into reflections
such that these reflections generate the Coxeter group $W$. Then Corollary~\ref{Corollary}  yields that  $w \in W$ is a quasi-Coxeter element if and only if
it has a reduced $T$-factorization  $w = t_1 \cdots t_n$ such that  the set of roots $R$ related to $\{t_1, \ldots ,t_n\}$ is a basis of the root lattice $L(\Phi)$ and such that the set
of coroots $R^\vee$ is a basis of the coroot lattice $L(\Phi^\vee)$.

 Theorem~\ref{thm:GenGroupLattice} should  help to clarify the notation of a  quasi-Coxeter element in the literature.
In \cite{BH16} it is an element such that the roots related to a reduced $T$-factorization of that element generate the root lattice.
Thus it is a consequence of Theorem~\ref{thm:GenGroupLattice} that quasi-Coxeter elements in the sense of \cite{BGRW16} are also quasi-Coxeter in \cite{BH16}, but not
vice versa.

In \cite[Corollary~6.10]{BGRW16} we characterized the parabolic subgroups of rank $n-1$ of a spherical dual Coxeter system $(W,T)$ of rank $n$ to be precisely those rank $n-1$ reflection subgroups that generate together with one additional reflection the whole group $W$.

\begin{theorem}\label{prop:CharParSub}\cite[Corollary~6.10]{BGRW16}
Let $(W,T)$ be a spherical dual Coxeter system of rank $n$ and $W'$ a reflection subgroup of rank $n-1$. Then $W'$ is a parabolic subgroup if and only if there exists $t \in T$ such that $\langle W', t \rangle = W$.
\end{theorem}

Since the completion of this work, it is come to our attention that
Taylor has shown the latter result already for finite complex reflection groups. His proof  uses a case-by-case argument (see \cite[Theorem 4.1]{Tay12}).

Here we show that in the Weyl group case
the reflection $t \in T$ so that $t$ and the rank-$(n-1)$ parabolic subgroup $P$ generate $W$ is unique up to $P$-conjugacy.

\begin{theorem}\label{prop:Conjugate}
	Let $(W,T)$ be a spherical crystallographic dual Coxeter system of rank $n$ and $P$ a  parabolic subgroup of $W$ of rank $n-1$.
	All the reflections $t \in T$ such that $W= \langle P, t\rangle$ form a single orbit under conjugation by $P$.
\end{theorem}

In fact we show more than stated in  Theorem~\ref{prop:Conjugate}. Theorem~\ref{FundamentalConjugation} is a direct consequence of the proof of Theorem~\ref{prop:Conjugate}.

\begin{theorem}\label{FundamentalConjugation}
Let $(W,T)$ be a spherical crystallographic dual Coxeter system of rank $n$ with root system $\Phi$. 	If  $\{\alpha_1, \ldots , \alpha_n\}$ is a simple system  for $\Phi$,  $P = \langle s_{\alpha_1}, \ldots , s_{\alpha_{n-1}} \rangle$  a  parabolic subgroup and
	$W = \langle P, s_\beta\rangle$ for some $\beta \in \Phi$, then there is $w \in P$ such that $w(\beta) = \alpha_n$.
    \end{theorem}

 \begin{remark}
 \begin{itemize}
\item[(a)]  Our first proof of Theorem~\ref{prop:Conjugate} was case-by-case.
Here we follow an idea by Vinberg \cite{Vi17} of a uniform proof of Theorem~\ref{prop:Conjugate}.
\item[(b)]
Observe that this theorem does not hold in general
if we remove the condition that $(W,T)$ is crystallographic, see the next example.
\end{itemize}
\end{remark}

\begin{example}
Consider a Coxeter system $(W,S)$ of type $H_3$. Let $S=\{ s_1, s_2, s_3 \}$ be such that $s_1$ and $s_3$ commute and  such that $s_1s_2$ and  $s_2s_3$ are of order  $5$ and $3$, respectively. Then for the parabolic subgroup $P:= \langle s_1, s_3 \rangle$ we obtain
$$
\langle P, s_2 \rangle = W = \langle  P, s_2s_1s_2 \rangle.
$$
But $s_2$ and $s_2s_1s_2$ are not conjugated under $P$.
\end{example}

Notice that the results presented here are used by the second named author to
show that the Hurwitz action on the set of reduced  factorizations of a quasi--Coxeter element is transitive in  dual affine Coxeter systems ~\cite{We17}.

The organization of this note is as follows. In Section \ref{sec2} we give all the necessary definitions. In Section \ref{Character} we prove Theorem \ref{thm:GenGroupLattice}; and in Section \ref{sec4} we provide a criterion for a set of roots to be a simple system and prove Theorem \ref{prop:Conjugate}.
\bigskip\\
{\bf Acknowledgement}   We like to thank Professor Ernest Vinberg for fruitful discussions with him as well as for explaining to us his idea of a uniform proof of Theorem~\ref{prop:Conjugate}.
 We also wish to thank the anonymous referee for helpful comments, and for making us aware of \cite{Tay12}.

\section{Spherical crystallographic dual Coxeter systems} \label{sec2}

 Let $(W,S)$ be a spherical Coxeter system of rank $n$, that is $W$ is a finite group and $S\subseteq W$ a set of involutions such that for $s,t \in S$ there exists $m(s,t) \in \NN$ with
 $m(s,s) = 1$, $m(s,t) = m(t,s)$ and $m(s,t) \geq 2$
 if $s \neq t$ so that  $W$ has the following presentation
 $$W = \langle S~|~ (st)^{m(s,t)} = 1,~s, t \in S\rangle.$$
 Let  $V$ be an $\R$-vector space with euclidean product $(\cdot \mid \cdot )$ and
   let $\alpha \in V$, $\alpha \neq 0$.  The reflection $s_\alpha$ is the orthogonal  transformation that fixes the hyperplane $H_\alpha:= \alpha^\bot $ pointwise and sends $\alpha$ to
 $-\alpha$, i.e.
 $$s_{\alpha} (v) = v -\frac{2(\alpha \mid v)}{(\alpha \mid \alpha)}\alpha$$ for all $v \in V$, see \cite{Hum90} or \cite{Bou02}.
There is a geometric representation  $\varphi: W \rightarrow \GL(V)$  of $W$  that sends $s \in S$  to the reflection
 $s_{\alpha_s}$ for some  vector $\alpha_s$ and that is faithful.
We identify $W$ with its image $\varphi(W)$ in $\GL(V)$.

 It follows that   $W$ leaves the form $(\cdot \mid \cdot)$ invariant, and that we may assume that $\{\alpha_s~|~s \in S\}$ generates the vector space $V$.
  Then  $\Phi:= \{w(\alpha_s)~|~s \in S, w \in W\}\subset V$ is a so called {\em root system } related
 to $(W,S)$, i.e. $\Phi$ is a finite set of nonzero vectors in $V$ such that
\begin{enumerate}
\item[(1)] $\spanr(\Phi) =V,$
\item[(2)] $s_{\alpha} ( \Phi ) = \Phi$ for all $\alpha \in \Phi$ and
\item[(3)] $\Phi \cap \RR \alpha = \{ \pm \alpha \}$ for all $\alpha \in \Phi.$
\end{enumerate}
Further, it follows that $W = W_\Phi:= \langle s_\alpha~|~\alpha \in \Phi\rangle$.

In this note we always assume that $\Phi$ is {\em crystallographic}, i.e.  $\frac{2(\alpha \mid \beta)}{(\alpha \mid \alpha)}$ is an integer for
 all $\alpha, \beta \in \Phi$. In this case the Coxeter system is also said to be crystallographic.
 
 For an irreducible crystallographic root system $\Phi$, the set 
$\{ (\alpha \mid \alpha) \mid \alpha \in \Phi \}$ has at most two elements (see \cite[Section 2.9]{Hum90}). If this set has only one element, we call $\Phi$ {\em simply laced}.

 Each root system $\Phi$ contains a so-called {\em simple system} $\Delta \subset \Phi$. That is a basis of $V$ such that the expression of each element in $\Phi$ in that basis has either all coefficients non-negative or
 non-positive integers.  In each simple system every two different roots have an obtuse dihedral angle \cite[Corollary 1.3]{Hum90}.
 Coxeter observed that this characterises the simple systems  in irreducible crystallographic root systems, see Lemma~\ref{Coxeter}.

 The simple systems have a second geometric description.
 The connected components of $V\setminus{\cup_{\alpha \in \Phi} H_\alpha}$ are open cones and are called {\em fundamental cones}. The hyperplanes bounding a fundamental cone are called
 {\em walls} of the fundamental cone.
 If   $\Delta = \{\alpha_1, \ldots, \alpha_n\} \subseteq \Phi$ is a simple system, then the hyperplanes $H_{\alpha_i}$, $1 \leq i \leq n$, are the walls of a fundamental cone
 $C  = \{x \in V~|~ (\alpha_i \mid x) > 0 $ for $1 \leq i \leq n\}$. In this way we get another characterization of the simple systems of $\Phi$. They are precisely the sets of roots for which there exists a fundamental cone such that each root is  orthogonal to one of the walls and  on the same side of the wall as the fundamental cone.

In this note we follow the dual approach to Coxeter systems.
A pair $(W,T)$ consisting of a group $W$ and a generating subset $T$ of $W$ is called {\em dual Coxeter system} of finite rank $n$ if there is a subset $S \subseteq T$ with $|S| = n$ such that
$(W,S)$ is a Coxeter system, and $T = \big\{ wsw^{-1} \mid w \in W, s \in S \big\}$
is the set of reflections for the Coxeter system $(W,S)$.
We then call~$(W,S)$ a {\em simple system} for $(W,T)$. Let $(W,T)$ be a dual Coxeter system.
If $S'\subseteq T$ is such that $(W,S')$ is a Coxeter system, then $\big\{ wsw^{-1} \mid w \in W, s \in S' \big\}=T$, see \cite[Lemma 3.7]{BMMN02}. Hence a set $S'\subseteq T$ is a simple system for $(W,T)$ if and only if $(W,S')$ is a Coxeter system. Note that the rank of $(W,T)$ is well-defined by \cite[Theorem 3.8]{BMMN02}. \\

Let $W'$ be a reflection subgroup of $W$, i.e. a subgroup of $W$ that is generated by reflections.
By \cite{Dye90} $(W', W' \cap T)$ is again a dual Coxeter system. The reflection subgroup generated by $\{s_1,\ldots,s_m\} \subseteq T$ is called a {\em parabolic subgroup} for $(W,T)$ if there is a simple system~$S = \{ s_1,\ldots,s_n \}$ for $(W,T)$ with $m \leq n$. This definition differs from the usual notion of a parabolic subgroup generated by a conjugate of a subset of a fixed simple system $S$ (see \cite[Section 1.10]{Hum90}). However, in \cite[Corollary 4.4]{BGRW16} it is shown that both definitions coincide for finite Coxeter groups.

An element $c\in W$ is called a {\em Coxeter element} if there exists a simple system $S=\{s_1,\dots, s_n\}$ for $(W,T)$ such that $c=s_1\cdots s_n$. An element $w \in W$ is called {\em quasi-Coxeter element} for $(W,T)$
if there exists a $T$-reduced factorization $w=t_1  \cdots  t_n$ such that $\langle t_1 , \ldots , t_n \rangle = W$. In particular,
every Coxeter element is a quasi-Coxeter element.

  Let as usual $\Phi^\vee$ be the set of coroots
  $$\alpha^\vee:= \frac{2\alpha}{(\alpha| \alpha)} ~\mbox{where}~\alpha \in \Phi.$$
 For a set of roots $R \subseteq \Phi$ we define $L(R):=\spanz(R)$ and $L(R^{\vee}):=\spanz(R^{\vee})$. We call $L(\Phi)$ and $L(\Phi^{\vee})$ the {\em root lattice} and the {\em coroot lattice}, respectively.

  Later we will need  the following two facts.

  \begin{Lemma}
  \cite[2.9]{Hum90} Let $\Phi$ be an irreducible and crystallographic root system of rank $n \geq 2$ and $\Delta:= \{ \alpha_1 , \ldots, \alpha_n\}$ be a simple system for $\Phi$. Then $\Delta^{\vee}= \{ \alpha_1^{\vee}, \ldots, \alpha_n^{\vee} \}$ is a simple system for $\Phi^{\vee}$ (and in particular a basis for $L(\Phi^{\vee})$).
  \end{Lemma}

  \begin{Lemma} \label{le:SumOfRoots}
  \cite[Ch. VI, 3, Cor. to Theorem 1]{Bou02} Let $\Phi$ be a crystallographic root system and $\alpha, \beta \in \Phi$ with $(\alpha \mid \beta)<0$ and $\alpha \neq - \beta$. Then $\alpha + \beta \in \Phi$.
  \end{Lemma}

 \section{Characterization of quasi-Coxeter elements: Proof of Theorem~\ref{thm:GenGroupLattice}}\label{Character}

\begin{proof}[{\bf Proof of Theorem~\ref{thm:GenGroupLattice}}]
The equivalence of $(a)$ and $(b)$ (not only for $\Phi$ simply laced) is already part of \cite[Proposition 5.10]{BGRW16}. Therefore it remains to prove the equivalence of $(b)$ and $(c)$.

We first show that $(b)$ implies $(c)$.
So assume $(b)$ and let $R:= \{\alpha_1, \ldots, \alpha_n\}$, and $t_i:=s_{\alpha_i}$, $1\leq i\leq k$. Let $T_{\Phi'}$ be the set of reflections in $W_{\Phi'}$. By \cite[Corollary 3.11 (ii)]{Dye90}, we have $T_{\Phi'}=\{wt_i w^{-1}~|~1\leq i\leq k, w\in W_{\Phi'}\}$. In particular every root in $\Phi'$ has the form $w(\alpha_i)$ for some $w\in W_{\Phi'}$, $1\leq i\leq k$. Since $W_{\Phi'}=\left\langle t_1,\dots, t_k\right\rangle$, we can write $w=t_{i_1}\cdots t_{i_m}$ with $1\leq i_j\leq k$ for each $1\leq j\leq m$. Since $\Phi'$ is crystallographic it follows that $w(\alpha_i)=t_{i_1}\cdots t_{i_m}(\alpha_i)$ is an integral linear combination of the $\alpha_j$'s, hence that $\Phi'\subseteq L(R)$. Since $R\subseteq \Phi'$ we get that $L(R)=L(\Phi')$.

By \cite[Ch. VI, Paragraph 1]{Bou02} there is an isomorphism $\varphi: W_{\Phi'} \stackrel{\sim}{\longrightarrow} W_{(\Phi')^{\vee}}$ with $\varphi(s_{\alpha})=s_{\alpha^{\vee}}$. Thus
$$W_{R^{\vee}}=\langle s_{\alpha_1^{\vee}}, \ldots , s_{\alpha_k^{\vee}} \rangle = \varphi(\langle s_{\alpha_1^{}}, \ldots , s_{\alpha_k^{}} \rangle)=  \varphi(W_{R})=
\varphi(W_{\Phi'}) =W_{(\Phi')^{\vee}}.$$ Using the same argumentation as in the last paragraph (now for $(\Phi')^{\vee}$ and $R^{\vee}$ instead of $\Phi'$ and $R$) we obtain $L((\Phi')^{\vee})=L(R^{\vee})$ as well, which shows $(c)$.

Now assume $(c)$, that is $L(R) = L(\Phi')$ and $L(R^{\vee})=L((\Phi')^{\vee})$. By \cite[Proposition 5.10]{BGRW16} it remains to show that if $\Phi''$ denotes the smallest root subsystem of $\Phi$ containing $R$, then $\Phi''=\Phi'$. Since $R \subseteq \Phi'' \subseteq \Phi'$  we get $L(R) \subseteq L(\Phi'') \subseteq
L(\Phi') = L(R)$, which yields
$L(R)=L(\Phi'') = L(\Phi')$. Let $\gamma \in \Phi'$. It remains to show that $\gamma \in \Phi''$. Since $L(\Phi')=  L(\Phi'')$, we have
$$\gamma = \sum_{i=1}^m \mu_i \beta_i$$
with $\mu_i \in \ZZ$ and $\beta_i \in \Phi''$. As $\beta_i \in \Phi''$ implies $- \beta_i \in \Phi''$, we may assume $\mu_i \in \ZZ_{>0}$. Therefore we can write
$$\gamma = \sum_{i=1}^m \beta_i$$
with $\beta_i \in \Phi''$, and we may assume that $m$ is minimal with that property (note that $\beta_i = \beta_j$ for $i \neq j $ is possible and that $m$ might have changed). We obtain
$$( \gamma \mid \gamma) = \sum_{i=1}^m (\beta_i \mid \beta_i) + \sum_{i \neq j} (\beta_i \mid \beta_j),$$
thus
$$1 = \sum_{i=1}^m \frac{(\beta_i \mid \beta_i)}{(\gamma \mid \gamma)} + \sum_{i \neq j} \frac{(\beta_i \mid \beta_j)}{(\gamma \mid \gamma)}.$$
Assume $\gamma \notin \Phi''$. This implies $m \geq 2$. If the root $\gamma$ is short, then $\sum_{i=1}^m \frac{(\beta_i \mid \beta_i)}{(\gamma \mid \gamma)} \geq 2$, hence $(\beta_i \mid \beta_j)<0$ for some $i \neq j$. By the minimality of $m$ we have $\beta_i \neq -\beta_j$. Therefore $\beta_i + \beta_j \in \Phi''$ by Lemma \ref{le:SumOfRoots}, contradicting the minimality of $m$. Thus $\gamma \in \Phi''$.

If the root $\gamma$ is long, then $\gamma^{\vee}$ is short. Since $L((\Phi')^{\vee})= L(R^{\vee})$, we can argue as before and obtain $\gamma^{\vee} \in (\Phi'')^{\vee}$. Thus $\gamma \in \Phi''$.
\end{proof}

\section{Minimal generating sets consisting of reflections} \label{sec4}

In this section we prove Theorem~\ref{prop:Conjugate}.  Throughout the section we assume as in the last section that $\Phi$ is a crystallographic root system.
We like to recall that in the crystallographic case simple systems are characterized as being set of roots whose pairwise dihedral angles are obtuse
as has been observed by Coxeter.
In \cite[Lemma~1]{DL11} Dyer and Lehrer proved,
if in a crystallographic root system $\Phi$ a subset of roots $R$ is linear independent and if the roots in $R$ have pairwise obtuse dihedral angles, then
$R$ is a simple system of a root subsystem of $\Phi$. Thus Lemma~\ref{Coxeter}
below  follows from Theorem~\ref{thm:GenGroupLattice} and \cite[Lemma~1]{DL11}.
We nevertheless present a short and  self contained proof of Lemma~\ref{Coxeter}.

 \begin{Lemma}\label{Coxeter}
Let $\Phi$ be a crystallographic root system of rank $n$ and $W=W_{\Phi}$. Let $R:= \{ \alpha_1, \ldots, \alpha_{n} \} \subseteq \Phi$  be a set of roots such that
\begin{itemize}
\item[(i)] $(\alpha_i \mid \alpha_j) \leq 0$ for $i \neq j$ and
\item[(ii)]  $W = \langle s_{\alpha_i}~|~ 1 \leq i \leq n\rangle$.
\end{itemize}
Then 	$R$ is a simple system for $\Phi$.
\end{Lemma}
\begin{proof}
We follow the ideas of the  proof given in \cite[Proposition 4.3]{Doe93}. By Theorem~\ref*{thm:GenGroupLattice} and (ii) we obtain that $R$ is  a basis of $V$.
Put
$$
C :=  \{ x \in V \mid (x \mid \alpha_i) > 0~\text{for all }i \in \{1, \ldots ,n\} \}.
$$

We claim that $C$ is a fundamental cone of $(W,S)$. Therefore we need to show that for every $\alpha \in \Phi$
either $(x \mid \alpha) > 0 $ for all $x \in C$  or $(x \mid \alpha) < 0 $ for all $x \in C$ . Assume there are $\alpha \in \Phi$,
$x_1, x_2 \in C$ such that $(x_1 \mid \alpha) < 0$ and $(x_2 \mid \alpha) > 0$. As $C$ is convex and $(\cdot \mid \alpha)$ continuous, we get $H_\alpha \cap C \neq \emptyset$.

Next we derive a contradiction to that fact.
As the roots in $R$ are pairwise obtuse and as $\Phi$ is crystallographic we have
$$
\left(\frac{\alpha_i}{|\alpha_i|}~ \Big\vert ~\frac{\alpha_j}{|\alpha_j|} \right) = - \cos \left(\frac{\pi}{m_{ij}}\right)
$$

where $m_{ij}$ is the order of $s_{\alpha_{i}} s_{\alpha_{j}} $ for $1 \leq i < j \leq n$ \cite[VI, 1.3, p. 161]{Bou02}.
Thus, $M = (m_{ij})$ is a Coxeter matrix and $W =\langle s_{\alpha_i}~|~ 1 \leq i \leq n\rangle$ is  geometric representation of the Coxeter system $(W(M), R)$ with Coxeter matrix $M$ by
\cite[V, 4.3]{Bou02}. As $H_\alpha \cap C \neq \emptyset$ implies $s_\alpha(C) \cap C \neq \emptyset$, we  get  $s_\alpha = \idop$
by \cite[V, 4.4, Theorem 1]{Bou02}, which is not possible.

This shows that $C$ is a fundamental cone and $R$ a simple system, see \cite[VI, 1.5, Theorem~2]{Bou02}.
\end{proof}

Notice that Lemma~\ref{Coxeter} implies 
that the 
diagram that Carter \cite{Carter} associates to an element of the Coxeter group $W$ has to contain a circle in the case that the element is
a quasi-Coxeter element that is not a Coxeter element, see \cite[Lemma~19]{Carter}.
\bigskip\\

\begin{proof}[{\bf Proof of Theorem~\ref{prop:Conjugate}}]
Let $\Delta:= \{\alpha_1, \ldots , \alpha_n\}$ be a simple system for $\Phi$ such that  $P = \langle s_{\alpha_1}, \ldots , s_{\alpha_{n-1}} \rangle$ and
	let $t \in T$ such that $W = \langle P, t\rangle$.
	As  $\Delta_P := \{\alpha_1, \ldots, \alpha_{n-1}\}$ is a simple system for a subsystem of $\Phi$, the dihedral angles between these roots are pairwise obtuse.

	 Let $\beta \in \Phi$ such that $s_\beta = t$. Further let $E$ be the cone of $V$ that is cut out by the hyperplanes $H_\alpha$ with $\alpha \in \Delta_P \cup \{\beta\}$. Then
$$
E = \{x \in V \mid (x \mid \alpha) >0 ~\text{for all}~ \alpha \in \Delta_P \cup \{\beta\} \}
$$
contains a fundamental cone.

	 Assume that $(\alpha_j \mid \beta) > 0$ for some $1 \leq j \leq n-1$. Then
	$$(\alpha_j \mid s_{\alpha_j}(\beta)) = (s_{\alpha_j}(\alpha_j) \mid \beta) = (-\alpha_j \mid \beta) = -(\alpha_j \mid \beta)  < 0$$
	and we replace $\beta$ by 	$s_{\alpha_j}(\beta)$.

	We claim that  the cone $F$ that is cut out by the new hyperplanes is contained in  $E$.  Let $x \in F$, that is $(x \mid \alpha_i) > 0$ for $1 \leq i \leq n-1$ and $(x \mid s_{\alpha_j}(\beta)) > 0$.
	Then $$(x \mid \beta) = (x \mid s_{\alpha_j}(s_{\alpha_j}(\beta)))  = (x \mid s_{\alpha_j}(\beta))  - \frac{2(s_{\alpha_j}(\beta) \mid \alpha_j) (x \mid \alpha_j)}{(\alpha_j \mid \alpha_j)} > 0$$
	as $(x \mid s_{\alpha_j}(\beta)) > 0$ and $(s_{\alpha_j}(\beta) \mid \alpha_j) < 0$, but $(x \mid \alpha_j) > 0$. Thus $F$ is contained in $E$.

    Next we show that this containment is proper.
    As $\Delta_P \cup \{\beta\}$ is linear independent by Theorem~\ref{thm:GenGroupLattice}, it follows that
$M:= \{ x\in V \mid (x \mid \alpha_i) > 0 ~\mbox{for}~ 1 \leq i \leq n-1\} \cap H_\beta \neq \emptyset $. Let $y \in M$. Then
    $$(y \mid s_{\alpha_j}(\beta)) = (s_{\alpha_j}(y)\mid \beta) = (y\mid \beta) -
    \frac{2(y\mid \alpha_j)(\alpha_j\mid \beta)}{(\alpha_j \mid \alpha_j)} <0$$
    as $(y\mid \beta) = 0$ and $(y\mid \alpha_j),(\alpha_j\mid \beta) >0$.
    Thus $y$ is in the closure of $E$, but not in the closure of $F$, which implies
    that $E \neq F$ and that therefore $F$ is a proper subset of $E$.

	As every sequence of cones $E_1 \supset E_2 \supset E_3 \cdots $ that are cut out
    by hyperplanes $H_\alpha$ with $\alpha \in \Phi$ is of finite length, this process will stop after finitely many steps. The set of roots that will be obtained
    is such that the pairwise dihedral angles between the roots are obtuse.
  Then Lemma~\ref{Coxeter} implies that the related cone is a fundamental
	cone and that the obtained set of roots $\{\alpha_1, \ldots, \alpha_{n-1}, \gamma\}$ is a simple system for $\Phi$. We got $\gamma$ from $\beta$ by conjugating it with elements
	of $P$.  This shows our claim.
\end{proof}
	\bigskip

The following is a consequence of  Theorem~\ref{prop:Conjugate}.

\begin{corollary}\label{Coro:prop:Conjugate}
Let $(W,T)$ be a spherical crystallographic dual Coxeter system of rank $n$, and let $r_1,t_1, \ldots, t_n \in T$ be reflections such that
$$W = \langle r_1, t_2, \ldots ,t_n\rangle = \langle t_1, \ldots ,t_n\rangle.$$
Then $g^{-1}r_1g = t_1$ for some $g \in \langle t_2, \ldots , t_n\rangle$.
\end{corollary}
\begin{proof} Theorem~\ref{prop:CharParSub} implies that $P:= \langle t_2, \ldots , t_n\rangle$ is a parabolic subgroup of $(W,T)$. The statement follows with
Theorem~\ref{prop:Conjugate}.
\end{proof}

\end{document}